\documentclass[9pt]{article}
\usepackage{amsfonts,amssymb,amsmath,comment}
\usepackage{graphicx,graphics,color}
 \makeatletter
 \@addtoreset{equation}{section}
 \makeatother

 \newcounter{enunciato}[section]

 \newtheorem{ittheorem}{Theorem}
 \newtheorem{itlemma}{Lemma}
 \newtheorem{itproposition}{Proposition}
 \newtheorem{itdefinition}{Definition}
 \newtheorem{itcorollary}{Corollary}
 \newtheorem{itconjecture}{Conjecture}

 \newenvironment{theorem}{\addtocounter{enunciato}{1}
 \begin{ittheorem}}{\end{ittheorem}}

 \newenvironment{lemma}{\addtocounter{enunciato}{1}
 \begin{itlemma}}{\end{itlemma}}

 \newenvironment{proposition}{\addtocounter{enunciato}{1}
 \begin{itproposition}}{\end{itproposition}}

\newcommand{\halmos}{\rule{1ex}{1.4ex}}

\newenvironment{proof}{\noindent {\em Proof}.\,\,}
{\hspace*{\fill}$\halmos$\medskip}

\def \qed {{\hspace*{\fill}$\halmos$\medskip}}

\def \ba {\begin{array}}
\def \ea {\end{array}}

\def \R {{\mathbb R}}

\def \O{\Omega}

\def \Leb {{\mathcal{L}}}

\begin{document}
\title{Torsional rigidity for cylinders with a Brownian fracture}

\author{{M. van den Berg}\\
School of Mathematics, University of Bristol\\
University Walk, Bristol BS8 1TW\\
United Kingdom\\
\texttt{mamvdb@bristol.ac.uk}\\
\\
{F. den Hollander}\\
Mathematical Institute, Leiden University\\
P.O.\ Box 9512, 2300 RA Leiden\\
The Netherlands\\
\texttt{denholla@leidenuniv.nl}}

\date{27 November 2017}

\maketitle

\begin{abstract}
We obtain bounds for the expected loss of torsional rigidity of a cylinder $\Omega_L=(-L/2,L/2)
\times \Omega\subset \R^3$ of length $L$ due to a Brownian fracture that starts at a random point
in $\Omega_L,$ and runs until the first time it exits $\Omega_L$. These bounds are expressed in
terms of the geometry of the cross-section $\Omega \subset \R^2$. It is shown that if $\Omega$ is
a disc with radius $R$, then in the limit as $L \rightarrow \infty$ the expected loss of torsional rigidity
equals $cR^5$ for some $c\in (0,\infty)$. We derive bounds for $c$ in terms of the expected
Newtonian capacity of the trace of a Brownian path that starts at the centre of a ball in $\R^3$
with radius $1,$ and runs until the first time it exits this ball.

\medskip\noindent
{\it AMS} 2000 {\it subject classifications.} 35J20, 60G50.

\medskip\noindent
{\it Acknowledgment.}
The authors acknowledge support by The Leverhulme Trust through International Network Grant
\emph{Laplacians, Random Walks, Bose Gas, Quantum Spin Systems}. FdH was also supported
by the Netherlands Organisation for Scientific Research (NWO) through Gravitation-grant
NETWORKS-024.002.003.

\end{abstract}

%%%%%%% SECTION 1 %%%%%%%%%%%%%%%%%%%%%%%%%%%%%%%

\section{Introduction}
\label{sec1}

In Section \ref{subsec1.1} we formulate the problem, in Section \ref{subsec1.2} we recall some
basic facts, in Section \ref{subsec1.3} we state our main theorems, and in Section \ref{subsec1.4}
we discuss these theorems and provide an outline of the remainder of the paper.

%%%

\subsection{Background and motivation}\label{subsec1.1}

Let $\Lambda$ be an open and bounded set in $\R^m$, with boundary $\partial \Lambda$
and Lebesgue measure $|\Lambda|$. Let $\Delta$ be the Laplace operator acting in
$\Leb^2(\R^m)$. Let $(\bar{\beta}(s),s \geq 0; \bar{\mathbb{P}}_x,x\in \R^m)$ be Brownian motion in $\R^m$
with generator $\Delta$. Denote the first exit time from $\Lambda$ by
\begin{equation*}%\label{e1}
\bar{\tau}(\Lambda) = \inf\{s \geq 0\colon\, \bar{\beta}(s)\in \R^m-\Lambda\},
\end{equation*}
and the expected lifetime in $\Lambda$ starting from $x$ by
\begin{equation*}%\label{e2}
v_{\Lambda}(x) = \bar{\mathbb{E}}_x[\bar{\tau}(\Lambda)], \qquad x\in \Lambda,
\end{equation*}
where $\bar{\mathbb{E}}_x$ denotes the expectation associated with $\bar{\mathbb{P}}_x$. The function
$v_{\Lambda}$ is the unique solution
of the equation
\begin{equation*}%\label{e3}
-\Delta v = 1, \qquad v\in H_0^1(\Lambda),
\end{equation*}
where the requirement $v\in H_0^1(\Lambda)$ imposes Dirichlet boundary conditions on
$\partial \Lambda$. The function $v_{\Lambda}$ is known as the \emph{torsion function} and found its origin
in elasticity theory. See  for example \cite{TG}. The \emph{torsional rigidity} $\mathcal{T}(\Lambda)$ of $\Lambda$ is defined
by
\begin{equation*}%\label{e4}
\mathcal{T}(\Lambda)=\int_{\Lambda}dx\,v_{\Lambda}(x).
\end{equation*}

Torsional rigidity plays a key role in many different parts of analysis. For example, the torsional rigidity of a
cross-section of a beam appears in the computation of the angular change when a beam of a given
length and a given modulus of rigidity is exposed to a twisting moment \cite{Bandle}, \cite{PSZ}.
It also arises in the calculation of the heat content of sets with time-dependent boundary conditions
\cite{MvdB}, in the definition of gamma convergence \cite{BB}, and in the study of minimal submanifolds
\cite{MP}. Moreover, $\mathcal{T}(\Lambda)/|\Lambda|$ equals the expected lifetime of Brownian motion
in $\Lambda$ when averaged with respect to the uniform distribution over all starting points $x\in\Lambda$.

Consider a finite cylinder in $\R^3$ of the form
\begin{equation*}%\label{e5}
\Omega_L = (-L/2,L/2) \times \Omega,
\end{equation*}
where $\Omega$ is an open and bounded subset of $\R^2$, referred to as the cross-section. It follows
from \cite[Theorem 5.1]{vdBBB} that
\begin{equation}\label{e6}
\mathcal{T}'(\Omega)L \geq \mathcal{T}(\Omega_L) = \mathcal{T}'(\Omega)L
- 4\mathcal{H}^2(\Omega)\lambda'_1(\Omega)^{-3/2},
\end{equation}
where $\mathcal{H}^2$ denotes the two-dimensional Hausdorff measure, $\lambda'_1(\Omega)$ is the
first eigenvalue of the two-dimensional Dirichlet Laplacian acting in $\Leb^2(\Omega)$, and $\mathcal{T}'
(\Omega)$ is the two-dimensional torsional rigidity of the planar set $\Omega$.

We observe that in \eqref{e6} the leading term is extensive, i.e., proportional to $L$, and that its coefficient
$\mathcal{T}'(\Omega)$ depends on the torsional rigidity of the cross-section $\Omega$. There is a
substantial literature on the computation of the two-dimensional torsional rigidity for given planar sets
$\Omega$. See, for example, \cite{TG} and \cite{R}. The finiteness of the cylinder induces a correction
that is at most $O(1)$.

Let $(\beta(s), s \geq 0; \mathbb{P}_x,x\in \R^m)$ be a Brownian motion, independent of $(\bar{\beta}(s),
s \geq 0; \bar{\mathbb{P}}_x,x\in \R^m)$,
and let
\begin{equation}\label{e7}
\tau(\Lambda)=\inf\{s \geq 0\colon\, \beta(s)\in \R^m-\Lambda\}.
\end{equation}
Denote its trace in $\Lambda$ up to the first exit time of $\Lambda$ by
\begin{equation}\label{e8}
\mathfrak{B}(\Lambda) = \{\beta(s)\colon\, 0 \leq s \leq \tau(\Lambda)\}.
\end{equation}
In this paper we investigate the effect of a Brownian fracture $\mathfrak{B}(\Omega_L)$ on the
torsional rigidity of $\Omega_L$. More specifically, we consider the random variable $\mathcal{T}
(\Omega_L-\mathfrak{B}(\Omega_L))$, and we investigate the expected loss of torsional rigidity
averaged over both the path $\mathfrak{B}(\Omega_L)$ and the starting point $y$, defined by
\begin{equation}\label{e9}
\mathfrak{T}(\Omega_L) = \frac{1}{|\Omega_L|}\int_{\Omega_L}dy\,
\mathbb{E}_y \big[\mathcal{T}(\Omega_L)
-\mathcal{T}(\Omega_L-\mathfrak{B}(\Omega_L))\big],
\end{equation}
where $\mathbb{E}_y$ denotes the expectation associated with $\mathbb{P}_y$.

%%%%%%%%%

\subsection{Preliminaries}\label{subsec1.2}

It is well known that the rich interplay between elliptic and parabolic partial differential
equations provides tools for linking various properties. See, for example, the monograph by
Davies \cite{davies}, and \cite{vdB2,vdB3,vdBBB,vdBFNT,vdBBdH} for more recent
results. As both the statements and the proofs of Theorems \ref{the1.1}, \ref{the1.2} and
\ref{the1.3} below rely on the connection between the torsion function, the torsional rigidity, and
the heat content, we recall some basic facts.

For an open set $\Lambda$ in $\R^m$ with boundary $\partial \Lambda$, we denote the
Dirichlet heat kernel by $p_{\Lambda}(x,y;t),\ x,y\in\Lambda,\ t>0$. The integral
\begin{equation}\label{e10}
u_{\Lambda}(x;t) = \int_{\Lambda} dy\,p_{\Lambda}(x,y;t), \qquad x \in \Lambda, \, t>0,
\end{equation}
is the unique weak solution of the heat equation
\begin{equation*}%\label{e11}
\frac{\partial u}{\partial t}(x;t) = \Delta u(x;t),\qquad x\in\Lambda,\, t>0,
\end{equation*}
with initial condition
\begin{equation*}%\label{e12}
\lim_{t \downarrow 0} u(\,\cdot\,;t)= 1\,\hbox{ in }\Leb^1(\Lambda),
\end{equation*}
and with Dirichlet boundary conditions
\begin{equation*}%\label{e13}
u(\,\cdot\,;t) \in H_0^1(\Lambda),\qquad t>0.
\end{equation*}
We denote the heat content of $\Lambda$ at time $t$ by
\begin{equation}\label{e14}
Q_{\Lambda}(t) = \int_{\Lambda} dx\,u_{\Lambda}(x;t)
= \int_{\Lambda} dx\,\int_{\Lambda} dy\, p_{\Lambda}(x,y;t), \qquad t>0.
\end{equation}
The heat content represents the amount of heat in $\Lambda$ at time $t$ when $\Lambda$
has initial temperature $1$ while $\partial\Lambda$ is kept at temperature $0$ for all $t>0$.
Since the Dirichlet heat kernel is non-negative and is monotone in $\Lambda$, we have
\begin{equation}\label{e15}
0 \leq p_{\Lambda}(x,y;t) \leq p_{\R^m}(x,y;t) = (4\pi t)^{-m/2}\,e^{-|x-y|^2/(4t)}.
\end{equation}
It follows from \eqref{e10} and \eqref{e15} that
\begin{equation*}%\label{e16}
0 \leq u_{\Lambda}(x;t) \leq 1, \qquad x \in \Lambda, \, t>0,
\end{equation*}
and that if $|\Lambda|<\infty$, then
\begin{equation}\label{e17}
0 \leq Q_{\Lambda}(t) \leq |\Lambda|, \qquad t>0.
\end{equation}
In the latter case we also have an eigenfunction expansion for the Dirichlet heat kernel in
terms of the Dirichlet eigenvalues $\lambda_1(\Lambda)\leq\lambda_2(\Lambda)\leq\dots,$
and a corresponding orthonormal set of eigenfunctions $\{\varphi_{\Lambda,1},
\varphi_{\Lambda,2},\dots\}$, namely,
\begin{equation*}%\label{e18}
p_{\Lambda}(x,y;t) = \sum_{j=1}^{\infty} e^{-t\lambda_j(\Lambda)}
\varphi_{\Lambda,j}(x)\varphi_{\Lambda,j}(y), \qquad x,y \in \Lambda, \, t>0.
\end{equation*}
We note that by \cite[p.63]{davies} the eigenfunctions are in $\Leb^p(\Lambda)$ for all $1\leq p
\leq \infty$. It follows from Parseval's formula that
\begin{equation}\label{e19}
Q_{\Lambda}(t) = \sum_{j=1}^{\infty} e^{-t\lambda_j(\Lambda)}
\left(\int_{\Lambda} dx\,\varphi_{\Lambda,j}(x)\right)^2
\leq e^{-t\lambda_1(\Lambda)}\sum_{j=1}^{\infty}
\left(\int_{\Lambda} dx\,\varphi_{\Lambda,j}(x)\right)^2
= e^{-t\lambda_1(\Lambda)}\,|\Lambda|, \qquad t>0,
\end{equation}
which improves upon \eqref{e17}. Since the torsion function is given by
\begin{equation*}%\label{e20}
v_{\Lambda}(x) = \int_{[0,\infty)} dt\,u_{\Lambda}(x;t), \qquad x \in \Lambda,
\end{equation*}
we have that
\begin{equation}\label{e21}
\mathcal{T}(\Lambda) = \int_{[0,\infty)} dt\,Q_\Lambda(t)
= \sum_{j=1}^{\infty} \lambda_j(\Lambda)^{-1}
\left(\int_\O dx\,\varphi_{\Lambda,j}(x)\right)^2.
\end{equation}

%%%

\subsection{Main theorems}\label{subsec1.3}

To state our theorems, we introduce the following notation. Two-dimensional quantities, such as
the heat content for the planar set $\Omega$, carry a superscript $'$. The Newtonian capacity of
a compact set $K\subset \R^3$ is denoted by $\textup{cap}(K)$. For $R,L>0$ we define
\begin{equation}\label{e21ext}
\begin{aligned}
D_R &= \{x'\in\R^2\colon\,|x'|<R\},\\
C_{L,R} &= (-L/2,L/2) \times D_R,\\
C_R &= C_{R,\infty}.
\end{aligned}
\end{equation}
For $x \in \R^3$ and $r>0$, we write $B(x;r) =\{y\in \R^3\colon\,|y-x|<r\}$.

\begin{theorem}\label{the1.1}
If $\Omega \subset \R^2$ is open and bounded, then
\begin{enumerate}
\item[\textup{(i)}]
\begin{equation}\label{e22}
0 \leq \mathcal{T}(\Omega_L)-\mathcal{T}'(\Omega)L
+ \frac{4}{\pi^{1/2}}\int_{[0,\infty)} dt\ t^{1/2}Q_{\Omega}'(t)
\leq \frac{8}{L}\,\lambda_1'(\Omega)^{-1}\mathcal{T}'(\Omega), \qquad L>0,
\end{equation}
\item[\textup{(ii)}]
\begin{equation}\label{e23}
\mathfrak{T}(\Omega_L)\le 6\lambda_1'(\Omega)^{-1/2}\mathcal{T}'(\Omega),
\qquad L > 0,
\end{equation}
\item[\textup{(iii)}]
\begin{equation}\label{e23a}
\limsup_{L\rightarrow\infty}\mathfrak{T}(\Omega_L)
\leq 4\lambda_1'(\Omega)^{-1/2}\mathcal{T}'(\Omega).
\end{equation}
\end{enumerate}
\end{theorem}

\begin{theorem}\label{the1.2}
If $\Omega= D_R$, then
\begin{equation}\label{e24}
\lim_{L\rightarrow \infty} \mathfrak{T}(C_{L,R})=cR^5, \qquad R>0,
\end{equation}
with
\begin{equation}\label{e25a}
\frac{67703\sqrt{79}-582194}{5059848192}\,\kappa \leq c \leq \frac{\pi}{2j_0},
\end{equation}
where $j_0=2.4048...$ is the first positive zero of the Bessel function $J_0$, and
\begin{equation*}%\label{e25b}
\kappa=\mathbb{E}_{0}\big[\textup{cap}\big(\mathfrak{B}(B(0;1))\big)\big].
\end{equation*}
\end{theorem}

We obtain better estimates when the Brownian fracture starts on the axis of the cylinder
$C_{L,R}$, with a uniformly distributed starting point. Let
\begin{equation}\label{e25c}
\mathfrak{C}(C_{L,R}) = \frac1L\int_{(-L/2,L/2)} dy_1\,
\mathbb{E}_{(y_1,0)}\Big[\mathcal{T}(C_{L,R})-\mathcal{T}\big(C_{L,R}-\mathfrak{B}(C_{L,R})\big)\Big].
\end{equation}

\begin{theorem}\label{the1.3}
If $\Omega= D_R$, then
\begin{equation}\label{e25d}
\lim_{L\rightarrow \infty}\mathfrak{C}(C_{L,R})=c'R^5, \qquad R>0,
\end{equation}
with
\begin{equation}\label{e25e}
 \frac{2867\sqrt{61}-21773}{303750}\, \kappa \leq c' \leq \frac{\pi}{4}\left(1+\frac{1}{j_0}\right).
\end{equation}
\end{theorem}

%%%%%%%

\subsection{Discussion and outline}\label{subsec1.4}

Theorem \ref{the1.1}(i) is a refinement of \eqref{e6}, while Theorems \ref{the1.1}(ii) and \ref{the1.1}(iii)
provide upper bounds for the expected loss of torsional rigidity. Theorem \ref{the1.2} gives a formula
for the expected loss of torsional rigidity in the special case where $\Omega$ is a disc with radius $R$.
Theorem \ref{the1.3} does the same when the fracture starts on the axis of the cylinder, with a uniformly
distributed starting point.

Computing the bounds in \eqref{e25a} numerically, we find that the upper bound is
$0.653$ and the lower bound is approximately $.386 \times 10^{-5}\kappa$. Since
$\kappa$ is bounded from above by $\textup{cap}(B(0;1))=4\pi$,
the left-hand side is at most $0.485 \times 10^{-4}$. Thus, the bounds are
at least $4$ orders of magnitude apart. It is not clear what the correct order of $c$ should be.
The bounds for $c'$ in Theorem \ref{the1.3} are at least two orders of magnitude apart.

The remainder of this paper is organised as follows. The proof of Theorem \ref{the1.1}
is given in Section \ref{sec2}, and uses the spectral representation of the heat kernel in
Section \ref{subsec1.2}. The proofs of Theorems \ref{the1.2} and \ref{the1.3} are given
in Section \ref{sec4}, and rely on a key proposition, stated and proved in Section \ref{sec3},
that provides a representation of the constants  $c$ and $c'$.

%%%%%%%%% SECTION 2 %%%%%%%%%%%%%%%%%%%%%%%%%

\section{Proof of Theorem \ref{the1.1}}\label{sec2}

{\it Proof of Theorem \ref{the1.1}\textup{(i)}.}
We use separation of variables, and write $x=(x_1,x'),\ y=(y_1,y')$, $x_1,y_1 \in \R$, $x',y' \in \R^2$.
Since the heat kernel factorises, we have
\begin{equation*}%\label{e25}
p_{\Omega_L}(x,y;t) = p_{(-L/2,L/2)}^{(1)}(x_1,y_1;t)\,p_{\Omega}'(x',y';t),
\qquad x,y\in \Omega_L,\ t>0,
\end{equation*}
where $p_{(-L/2, L/2)}^{(1)}(x_1,y_1;t)$ is the one-dimensional Dirichlet heat kernel for the
interval $(-L/2,L/2)$, and $p_{\Omega}'(x',y';t)$ is the two-dimensional Dirichlet heat kernel
for the planar set $\Omega$. By integrating over $\Omega_L$, we see that the heat content also
factorises,
\begin{equation}\label{e26}
Q_{\Omega_L}(t) = Q^{(1)}_{(- L/2, L/2)}(t)\,Q'_{\Omega}(t), \qquad t>0,
\end{equation}
where $Q^{(1)}_{(-L/2,L/2)}$ is the one-dimensional heat content for the interval
$(-L/2, L/2)$, and $Q_\Omega'$ is the two-dimensional heat content for the planar
set $\Omega$. In
\cite{vdBBB} it was shown that
\begin{equation}\label{e27}
L-\frac{4t^{1/2}}{\pi^{1/2}} \leq Q^{(1)}_{(-L/2,L/2)}(t)
\leq L-\frac{4t^{1/2}}{\pi^{1/2}} + \frac{8t}{L},
\qquad t>0.
\end{equation}
Combining \eqref{e21}, \eqref{e26} and \eqref{e27}, we have
\begin{align}\label{e28}
\mathcal{T}(\Omega_L)
&=\int _{[0,\infty)} dt\ Q_{\Omega_L}(t)
\leq \int _{[0,\infty)} dt\ \bigg(L-\frac{4t^{1/2}}{\pi^{1/2}}
+ \frac{8t}{L}\bigg) Q'_{\Omega}(t)\nonumber \\
&= L\mathcal{T}'(\Omega)- \frac{4}{\pi^{1/2}}\int_{[0,\infty)} dt\ t^{1/2}Q_{\Omega}'(t)
+\frac{8}{L}\int_{[0,\infty)} dt\ t\, Q'_{\Omega}(t).
\end{align}
To bound the third term in the right-hand side of \eqref{e28}, we use the identities in \eqref{e19}
and \eqref{e21} to obtain
\begin{align}\label{e29}
\int_{[0,\infty)} dt\ t\  Q'_{\Omega}(t)
&= \int_{[0,\infty)} dt\ t\sum_{j=1}^{\infty} e^{-t\lambda_j'(\Omega)}
\left(\int_{\Omega} dx\,\varphi_{\Omega,j}(x)\right)^2
= \sum_{j=1}^{\infty}\lambda_j'(\Omega)^{-2}
\left(\int_{\Omega} dx\,\varphi_{\Omega,j}(x)\right)^2 \nonumber \\
&\leq \lambda_1'(\Omega)^{-1}\sum_{j=1}^{\infty}\lambda_j'(\Omega)^{-1}
\left(\int_{\Omega} dx\,\varphi_{\Omega,j}(x)\right)^2
=\lambda_1'(\Omega)^{-1}\mathcal{T}'(\Omega).
\end{align}
This completes the proof of the right-hand side of \eqref{e22}. The left-hand side of \eqref{e22}
follows from \eqref{e21}, \eqref{e26} and the first inequality in \eqref{e27}. \qed

\medskip\noindent
{\it Proof of Theorem \ref{the1.1}\textup{(ii)}.}
Since $\Omega_L\subset \R\times\Omega$, we have that $v_{\Omega_L}(x_1,x')\le v_{\R\times \Omega}(x_1,x')
=v_{\Omega}'(x')$. Hence
\begin{equation}\label{e30}
\mathcal{T}(\Omega_L)\le \int_{(-L/2,L/2)} dx_1\ \int_{\Omega}dx'\ v_{\Omega}'(x')
= L\mathcal{T}'(\Omega).
\end{equation}
To prove the upper bound in \eqref{e23}, we recall \eqref{e9} and combine \eqref{e30} with a lower
bound for $\mathbb{E}_y[(\mathcal{T}(\Omega_L-\mathfrak{B}(\Omega_L))]$.
We observe that, for the Brownian motion defining $\mathfrak{B}(\Omega_L)$ (recall \eqref{e7}
and \eqref{e8}) with starting point $\beta(0)=(\beta_1(0),\beta'(0))$,
\begin{equation*}%\label{e31}
\tau(\Omega_L) \leq \tau'(\Omega)=\inf\{s\ge 0\colon\, \beta'(s)\notin \Omega\}.
\end{equation*}
Hence
\begin{equation*}%\label{e32}
\mathfrak {B}(\Omega_L) \subset \left[\max\left\{-\frac{L}{2},\min_{0\leq s \leq \tau'(\Omega)} \beta_1(s)\right\},
\min\left\{\frac{L}{2}, \max_{0 \leq s \leq \tau'(\Omega)} \beta_1(s)\right\}\right] \times \Omega.
\end{equation*}
Therefore $\Omega_L-\mathfrak{B}(\Omega_L)$ is contained in the union of at most two cylinders with
cross-section $\Omega$ and with lengths $\big(L/2+\min_{0 \leq s \leq \tau'(\Omega)}
\beta_1(s)\big)_+$ and $\big(L/2-\max_{0 \leq s \leq \tau'(\Omega)}\beta_1(s)\big)_+$,
respectively. For each of these cylinders we apply the lower bound in Theorem \ref{the1.1}(i),
taking into account that the total length of these cylinders is bounded from below by
$L-\big(\max_{0 \leq s \leq \tau'(\Omega)}\beta_1(s)-\min_{0 \leq s \leq \tau'(\Omega)}
\beta_1(s)\big)$. This gives
\begin{equation}\label{e33}
\mathcal{T}\big(\Omega_L-\mathfrak{B}(\Omega_L)\big) \geq
\left(L-\left(\max_{0 \leq s \leq \tau'(\Omega)} \beta_1(s)
-\min_{0 \leq s\le \tau'(\Omega)} \beta_1(s)\right)\right) \mathcal{T}'(\Omega)
- \frac{8}{\pi^{1/2}}\int_{[0,\infty)}dt\ t^{1/2}Q_{\Omega}'(t).
\end{equation}
With obvious abbreviations, by the independence of the Brownian motions $B_1$ and $B'$, we
have that $\mathbb{E}_{(y_1,y')}= \mathbb{E}_{y_1}\otimes \mathbb{E}_{y'}$. For the expected
range of one-dimensional Brownian motion it is known that (see, for example, \cite{DSY})
\begin{equation}\label{e34}
\mathbb{E}_{y_1}\left[\max_{0 \leq s\leq \tau'(\Omega)}\beta_1(s)
-\min_{0 \leq s \leq \tau'(\Omega)}\beta_1(s)\right]
= \frac{4\tau'(\Omega)^{1/2}}{\pi^{1/2}}.
\end{equation}
Furthermore,
\begin{align}\label{e35}
\mathbb{E}_{y'}\big[\tau'(\Omega)^{1/2}\big]
&= \int_{[0,\infty)} d\tau\,\tau^{1/2}\,\mathbb{P}_{y'}\big(\tau'(\Omega)\in d\tau\big)
= -\int_{[0,\infty)} d\tau\,\tau^{1/2}\bigg(\frac{d}{d\tau}\mathbb{P}_{y'}\big(\tau'(\Omega)>\tau\big)\bigg)
\nonumber \\
&= \frac12\int_{[0,\infty)} d\tau\,\tau^{-1/2}\, \mathbb{P}_{y'}\big(\tau'(\Omega)>\tau\big)
= \frac12 \int_{[0,\infty)}d\tau\,\tau^{-1/2}\int_{\Omega} dz'\,p_{\Omega}'(y',z';\tau).
\end{align}
Therefore, by \eqref{e14} and Tonelli's theorem,
\begin{equation}\label{e36}
\int_{\Omega}dy'\, \mathbb{E}_{y'}\big[\tau'(\Omega)^{1/2}\big]
= \frac12 \int_{[0,\infty)} d\tau\,\tau^{-1/2}Q_{\Omega}'(\tau).
\end{equation}
So with $|\Omega_L|/L = \mathcal{H}^2(\Omega)$,
\begin{equation}\label{e37}
\frac{1}{|\Omega_L|}\int_{\Omega_L} dy\, \mathbb{E}_{y}\big[\tau'(\Omega)^{1/2}\big]
= \frac{1}{2\mathcal{H}^2(\Omega)} \int_{[0,\infty)} d\tau\ \tau^{-1/2} Q_{\Omega}'(\tau).
\end{equation}
Combining \eqref{e9}, \eqref{e30}, \eqref{e33} and \eqref{e37}, we obtain
\begin{equation}\label{e38}
\mathfrak{T}(\Omega_L) \leq \frac{8}{\pi^{1/2}}\int_{[0,\infty)}dt\ t^{1/2}Q_{\Omega}'(t)
+ \left(\frac{2}{\pi^{1/2}\mathcal{H}^2(\Omega)} \int_{[0,\infty)} d\tau\
\tau^{-1/2} Q_{\Omega}'(\tau)\right)\,\mathcal{T}'(\Omega).
\end{equation}
The second integral in the right-hand side of \eqref{e38} can be bounded from above
using \eqref{e19}. This gives that
\begin{equation}\label{e40}
\frac{2}{\pi^{1/2} \mathcal{H}^2(\Omega)} \int_{[0,\infty)} d\tau\ \tau^{-1/2}Q_{\Omega}'(\tau)
\leq \frac{2}{\pi^{1/2}} \int_{[0,\infty)} d\tau\,\tau^{-1/2}\,e^{-\tau\lambda_1'(\Omega)}
= 2\lambda_1'(\Omega)^{-1/2}.
\end{equation}
Via a calculation similar to the one in \eqref{e29}, we obtain that
\begin{equation}\label{e41}
\frac{8}{\pi^{1/2} }\int_{[0,\infty)}dt\ t^{1/2}Q_{\Omega}'(t) \leq
4\lambda_1'(\Omega)^{-1/2}\,\mathcal{T}'(\Omega).
\end{equation}
Combining \eqref{e38}, \eqref{e40} and \eqref{e41}, we arrive at \eqref{e23}.
\qed

\medskip\noindent
{\it Proof of Theorem \ref{the1.1}\textup{(iii)}.}
If we use the upper bound in \eqref{e22} instead of the upper
bound in \eqref{e30}, then we obtain that
\begin{equation*}%\label{e42}
\mathfrak{T}(\Omega_L) \leq 4\lambda_1'(\Omega)^{-1/2}\,\mathcal{T}'(\Omega)
+ 8L^{-1}\lambda_1'(\Omega)^{-1}\,\mathcal{T}'(\Omega).
\end{equation*}
This in turn implies \eqref{e23a}.
\qed

%%%%%%%% SECTION 3 %%%%%%%%%%%%%%%%%%%%%

\section{Key proposition}\label{sec3}

The proofs of Theorems \ref{the1.2} and \ref{the1.3} rely on the following proposition which
states formulae for the constants $c$ in \eqref{e24} and $c'$ in \eqref{e25d}, respectively.
We recall definitions \eqref{e9}, \eqref{e21ext} and \eqref{e25c}.

\begin{proposition}\label{prop1}
If $\Omega=D_R$, then
\begin{equation}\label{e43}
\lim_{L\rightarrow\infty} \mathfrak{T}(C_{L,R}) = cR^5,
\qquad
\lim_{L\rightarrow \infty} \mathfrak{C}(C_{L,R}) = c'R^5,
\qquad R>0,
\end{equation}
with
\begin{equation}\label{e43ext}
\begin{aligned}
c &= \frac{1}{\pi} \int_{D_1}dy'\,\mathbb{E}_{(0,y')} \left[\int_{C_1} dx\,
\Big(v_{C_1}(x)-v_{C_1-\mathfrak{B}(C_1)}(x)\Big)\right],\\
c' &=\mathbb{E}_{(0,0)}\left[\int_{C_1} dx\,
\Big(v_{C_1}(x)-v_{C_1-\mathfrak{B}(C_1)}(x)\Big)\right].
\end{aligned}
\end{equation}
\end{proposition}

\begin{proof}
The proof for $\mathfrak{T}(C_{L,R})$ comes in 10 Steps.

\medskip\noindent
{\bf 1.}
By \eqref{e9},
\begin{align}\label{e43a}
\mathfrak{T}(C_{L,R})
&= \frac{1}{\pi R^2L}\int_{C_{L,R}}dy\, \, \mathbb{E}_y \left[\int_{C_{L,R}} dx\,
\big(v_{C_{L,R}}(x)-v_{C_{L,R}-\mathfrak{B}(C_{L,R})}(x)\big)\right].
\end{align}
We observe that $x \mapsto v_{C_{L,R}}(x)-v_{C_{L,R}-\mathfrak{B}(C_{L,R})}(x)$ is harmonic on
$C_{L,R}-\mathfrak{B}(C_{L,R})$, is non-negative, and equals $0$ for $x \in \partial C_{L,R}$. By
Lemma \ref{lem4} in Appendix \ref{app}, $N \mapsto v_{C_{N,R}}(x)-v_{C_{N,R}-\mathfrak{B}(C_{L,R})}(x)$ is
increasing on $[L,\infty)$, and bounded by $\tfrac14 R^2$ uniformly in $x$. Therefore
\begin{align}\label{e43b}
v_{C_{L,R}}(x)-v_{C_{L,R}-\mathfrak{B}(C_{L,R})}(x)
&\leq \lim_{N \rightarrow \infty} \big(v_{C_{N,R}}(x)-v_{C_{N,R}-\mathfrak{B}(C_{L,R})}(x)\big)\nonumber \\
&=\lim_{N \rightarrow \infty}v_{C_{N,R}}(x)
-\lim_{N \rightarrow \infty}v_{C_{N,R}-\mathfrak{B}(C_{N,R})}(x)\nonumber \\
& =v_{C_{R}}(x)-v_{C_{R}-\mathfrak{B}(C_{L,R})}(x)\nonumber \\ &
\leq v_{C_{R}}(x)-v_{C_{R}-\mathfrak{B}(C_{R})}(x), \quad x\in C_{L,R}-\mathfrak{B}(C_{L,R}).
\end{align}
The last inequality in \eqref{e43b} follows from the domain monotonicity of the torsion function. Inserting
\eqref{e43b} into \eqref{e43a}, we get
\begin{equation}\label{e43c}
\mathfrak{T}(C_{L,R}) \leq \frac{1}{\pi R^2L}\int_{C_{L,R}}dy\, \int_{C_{R}}dx\,\,
\mathbb{E}_y\left[\Big(v_{C_{R}}(x)-v_{C_{R}-\mathfrak{B}(C_{R})}(x)\Big)\right].
\end{equation}
Since $v_{C_R}(x)$ is independent of $x_1$, we have $v_{C_R}(x)=v_{C_R}(x-(y_1,0))$ and so
\begin{equation}\label{e43d}
\mathbb{E}_y\left[v_{C_R}(x)\right] = \mathbb{E}_{(0,y')}\left[v_{C_R}(x-(y_1,0))\right].
\end{equation}
Since the stopping time $\tau(C_R-\mathfrak{B}(C_{R}))$ is independent of $y_1$, we also see that
\begin{equation}\label{e43e}
\mathbb{E}_y\big[v_{C_{R}-\mathfrak{B}(C_{R})}(x)\big]
= \mathbb{E}_{(0,y')} \big[v_{C_R-\mathfrak{B}(C_{R})}(x-(y_1,0))\big].
\end{equation}
Combining \eqref{e43c}, \eqref{e43d} and \eqref{e43e}, we obtain
\begin{align*}%\label{e43f}
\mathfrak{T}(C_{L,R})
&\leq \frac{1}{\pi R^2L}\int_{C_{L,R}}dy\, \mathbb{E}_{(0,y')}\left[
\int_{C_{R}}dx\,\Big(v_{C_R}(x-(y_1,0))-v_{C_R-\mathfrak{B}(C_{R})}(x-(y_1,0))\Big)\right]\nonumber \\
&=\frac{1}{\pi R^2L}\int_{C_{L,R}}dy\, \mathbb{E}_{(0,y')}\left[
\int_{C_{R}}dx\,\Big(v_{C_R}(x)-v_{C_R-\mathfrak{B}(C_{R})}(x)\Big)\right]\nonumber \\
&=\frac{1}{\pi R^2}\int_{D_R}dy'\, \mathbb{E}_{(0,y')}\left[
\int_{C_{R}}dx\,\Big(v_{C_R}(x)-v_{C_R-\mathfrak{B}(C_{R})}(x)\Big)\right].
\end{align*}
We conclude that
\begin{equation*}%\label{e43g}
\limsup_{L\rightarrow\infty}\mathfrak{T}(C_{L,R}) \leq \frac{1}{\pi R^2}\int_{D_R} dy'\,
\mathbb{E}_{(0,y')}\left[ \int_{C_R}dx\,\big(v_{C_R}(x)-v_{C_R-\mathfrak{B}(C_R)}(x)\big)\right].
\end{equation*}
Scaling each of the space variables $y'$ and $x$ by a factor $R$, we gain a factor $R^5$ for the
respective integrals with respect to $y'$ and $x$. Furthermore, scaling the torsion functions
$v_{C_R}$ and $v_{C_R-\mathfrak{B}(C_R)}$, we gain a further factor $R^2$. This completes
the proof of the upper bound for $c$.

\medskip\noindent
{\bf 2.}
To obtain the lower bound for $c$, we define $\tilde{L}=\{x\in \R^3\colon\, x_1=\pm L/2\}$ and
\begin{equation*}%\label{e43h}
\tilde{C}_{L,R} = \left\{(x_1,x') \in C_R\colon\,-\frac{L}{2}+(RL)^{1/2}<x_1<\frac{L}{2}-(RL)^{1/2}\right\},
\qquad L\geq 4R.
\end{equation*}
Then, with $\textbf{1}$ denoting the indicator function, we have that
\begin{align}\label{e43i}
\mathfrak{T}(C_{L,R})
&\geq \frac{1}{\pi R^2L}\int_{\tilde{C}_{L,R}}dy\, \,
\mathbb{E}_y \left[\int_{C_{L,R}}dx\, \Big(v_{C_{L,R}}(x)
-v_{C_{L,R}-\mathfrak{B}(C_{L,R})}(x)\Big)\right]\nonumber \\
&\ge \frac{1}{\pi R^2L}\int_{\tilde{C}_{L,R}}dy\, \, \mathbb{E}_y \left[\textbf {1}_{\{\mathfrak{B}(C_{L,R})
\cap \tilde{L}=\emptyset\}}\int_{C_{L,R}}dx\, \Big(v_{C_{L,R}}(x)-v_{C_{L,R}
-\mathfrak{B}(C_{L,R})}(x)\Big)\right]\nonumber \\
&=\frac{1}{\pi R^2L}\int_{\tilde{C}_{L,R}}dy\, \, \mathbb{E}_y\left[\textbf {1}_{\{\mathfrak{B}(C_{L,R})
\cap \tilde{L}=\emptyset\}}\int_{C_{L,R}}dx\, \Big(v_{C_{L,R}}(x)-v_{C_{L,R}
-\mathfrak{B}(C_{R})}(x)\Big)\right]\nonumber \\
&=\frac{1}{\pi R^2L}\int_{\tilde{C}_{L,R}}dy\, \, \mathbb{E}_y
\left[\int_{C_{L,R}}dx\, \Big(v_{C_{L,R}}(x)-v_{C_{L,R}-\mathfrak{B}(C_{R})}(x)\Big)\right]-A_1,
\end{align}
and
\begin{align}\label{e43j}
A_1
&= \frac{1}{\pi R^2L}\int_{\tilde{C}_{L,R}} dy\, \, \mathbb{E}_y \left[\textbf {1}_{\{\mathfrak{B}(C_{L,R})
\cap \tilde{L}\ne \emptyset\}}\int_{C_{L,R}} dx\,
\Big(v_{C_{L,R}}(x)-v_{C_{L,R}-\mathfrak{B}(C_{L,R})}(x)\Big)\right] \nonumber \\
&\leq \frac{1}{\pi R^2L} \int_{\tilde{C}_{L,R}} dy\, \mathbb{E}_y \left[\textbf {1}_{\{\mathfrak{B}(C_{L,R})
\cap \tilde{L}\ne \emptyset\}}\right]\int_{C_{L,R}}dx\, v_{C_{L,R}}(x)\nonumber \\
&\leq \frac{R^2}{8}\int_{\tilde{C}_{L,R}}dy\, \mathbb{P}_y \big(\mathfrak{B}(C_{L,R})
\cap \tilde{L}\ne \emptyset\big)\nonumber \\ &
\leq \frac{\pi R^4L}{8}\sup_{y\in \tilde{C}_{L,R}}\mathbb{P}_y\big(\theta(\tilde{L}) \leq \tau(C_R)\big),
\end{align}
where
\begin{equation*}%\label{e43jj}
\theta(K) = \inf\{s \geq 0\colon\, \beta(s)\in K\}
\end{equation*}
denotes the first entrance time of $K$. The penultimate inequality in \eqref{e43j} uses the two bounds
$\int_{C_{L,R}} dx\,v_{C_{L,R}}(x) \le\int_{C_{L,R}} dx\,v_{C_{R}}(x)= \tfrac18\pi R^4L$ and
$|\tilde{C}_{L,R}| \le \pi R^2L$.

\medskip\noindent
{\bf 3.}
The following lemma gives a decay estimate for the supremum in the right-hand side of \eqref{e43j}
and implies that $\lim_{L\rightarrow\infty} A_1=0$.

\begin{lemma}\label{lem3}
\begin{equation}\label{e43k}
\sup_{y\in \tilde{C}_{L,R}} \mathbb{P}_y\big(\theta(\tilde{L})
\leq \tau(C_R)\big) \leq (j_0+1)\pi^{1/2}\, e^{-j_0L^{1/2}/(2R^{1/2})},\qquad L\geq 4R.
\end{equation}
\end{lemma}

\begin{proof}
First observe that the distance of $y$ to $\tilde{L}$ is bounded from below by $(LR)^{1/2}.$ Therefore
\begin{equation}\label{e43l}
\mathbb{P}_y\big(\theta(\tilde{L})\le \tau(C_R)\big)
\leq \mathbb{P}_{(0,y')}\left(\max_{0 \leq s \leq \tau'(C_R)}|\beta_1(s)| \geq (LR)^{1/2}\right).
\end{equation}
By \cite[(6.3),Corollary 6.4]{vdBD},
\begin{equation}\label{e43m}
\mathbb{P}^{(1)}_0\left(\max_{0\leq s\leq t} |\beta_1(s)| \geq R\right) \leq 2^{3/2}e^{-R^2/(8t)}.
\end{equation}
Combining \eqref{e43l} and \eqref{e43m} with the independence of $\beta_1$ and $\beta'$, we obtain
via an integration by parts,
\begin{align}\label{e43n}
\mathbb{P}_y\big(\theta(\tilde{L})
&\le \tau(C_R)\big)\le 2^{3/2}\int_{[0,\infty)} d\tau\, \bigg(\frac{\partial}{\partial \tau}
\mathbb{P}_{y'}\big(\tau'(D_R)>\tau\big)\bigg)\,e^{-LR/(8\tau)}\nonumber \\
&=\frac{LR}{2^{3/2}}\int_{[0,\infty)}\frac{d\tau}{\tau^2}\,
\mathbb{P}_{y'}\big(\tau'(D_R)>\tau\big)\,e^{-LR/(8\tau)}.
\end{align}
By the Cauchy-Schwarz inequality, the semigroup property of the heat kernel, the eigenfunction
expansion of the heat kernel, and the domain monotonicity of the heat kernel, we have that
\begin{align}\label{e43o}
\mathbb{P}_{y'}\big(\tau'(D_R)>\tau\big)
&=\int_{D_R}dz'\, p'_{D_R}(z',y';\tau)\nonumber \\ &
\leq (\pi R^2)^{1/2}\bigg(\int_{D_R}dz'\, (p'_{D_R}(z',y';\tau))^2\bigg)^{1/2}\nonumber \\
&= (\pi R^2)^{1/2}\big(p'_{D_R}(y',y';2\tau)\big)^{1/2}\nonumber \\
&= (\pi R^2)^{1/2}\bigg(\sum_{j=1}^{\infty} e^{-2\tau\lambda'_j(D_R)}
\big(\varphi_{D_R,j}'(y')\big)^2\bigg)^{1/2}\nonumber \\
&\leq (\pi R^2)^{1/2}e^{-\tau\lambda'_1(D_R)/2}
\bigg(\sum_{j=1}^{\infty} e^{-\tau\lambda'_j(D_R)}\big(\varphi_{D_R,j}'(y')\big)^2\bigg)^{1/2}\nonumber \\
&= (\pi R^2)^{1/2}e^{-\tau\lambda'_1(D_R)/2}\big(p'_{D_R}(y',y';\tau)\big)^{1/2}\nonumber \\
&\leq (\pi R^2)^{1/2}e^{-\tau\lambda'_1(D_R)/2}\big(p'_{\R^2}(y',y';\tau)\big)^{1/2}\nonumber \\ &
= \frac{Re^{-j_0^2\tau/(2R^2)}}{(4\tau)^{1/2}}.
\end{align}
Combining \eqref{e43n} and \eqref{e43o}, and changing variables twice, we arrive at
\begin{align}\label{e43p}
\mathbb{P}_y\big(\theta(\tilde{L}) \le \tau(C_R)\big)
&\leq \frac{LR^2}{2^{5/2}}\int_{[0,\infty)}\frac{d\tau}{\tau^{5/2}}\,
e^{-j_0^2\tau/(2R^2)-LR/(8\tau)}\nonumber \\
&= \frac{j_0^{3/2}L^{1/4}}{2R^{1/4}}\int_{[0,\infty)}\frac{d\tau}{\tau^{5/2}}\,
e^{-j_0L^{1/2}(\tau+\tau^{-1})/(4R^{1/2})}\nonumber \\
&=\frac{j_0^{3/2}L^{1/4}}{R^{1/4}}\int_{[0,\infty)}\frac{d\tau}{\tau^{4}}\,
e^{-j_0L^{1/2}(\tau^2+\tau^{-2})/(4R^{1/2})}\nonumber \\&
=\pi^{1/2}j_0\bigg(1+\frac{2R^{1/2}}{j_0L^{1/2}}\bigg)e^{-j_0L^{1/2}/(2R^{1/2})}.
\end{align}
The last equality follows from \cite[3.472.4]{GR}. This proves \eqref{e43k} because $L\geq 4R$.
\end{proof}

\medskip\noindent
{\bf 4.}
We write the double integral in the right-hand side of \eqref{e43i} as $B_1+B_2$, where
\begin{equation}\label{e43q}
B_1=\frac{1}{\pi R^2L}\int_{\tilde{C}_{L,R}}dy\, \,
\mathbb{E}_y \left[{\textbf{1}}_{\{\mathfrak{B}(C_{R}) \cap \hat{L}=\emptyset\}}
\int_{C_{L,R}}dx\, \Big(v_{C_{L,R}}(x)-v_{C_{L,R}-\mathfrak{B}(C_{R})}(x)\Big)\right],
\end{equation}
\begin{equation*}%\label{e43r}
B_2=\frac{1}{\pi R^2L}\int_{\tilde{C}_{L,R}}dy\, \,
\mathbb{E}_y \left[{\textbf{1}}_{\{\mathfrak{B}(C_{L,R})\cap \hat{L}\ne\emptyset\}}
\int_{C_{L,R}}dx\, \Big(v_{C_{L,R}}(x)-v_{C_{L,R}-\mathfrak{B}(C_{R})}(x)\Big)\right],
\end{equation*}
with
\begin{equation*}%\label{e43s}
\hat{L}=\pm\frac{L}{2}\mp\frac{(RL)^{1/2}}{2}.
\end{equation*}
We have that
\begin{align}\label{e43t}
B_2&\le\frac{1}{\pi R^2L}\int_{\tilde{C}_{L,R}}dy\,
\mathbb{P}_y\big(\mathfrak{B}(C_{R})\cap\hat{L}\ne \emptyset\big)
\int_{C_{L,R}}dx\, v_{C_{R}}(x)\nonumber \\
&\leq \frac{\pi R^4L}{8}\sup_{y\in \tilde{C}_{L,R}}
\mathbb{P}_y\big(\tau({\hat L})\le \tau(C_R)\big).
\end{align}
The distance from any $y\in \tilde{C}_{L,R}$ to $\hat{L}$ is bounded from below by $(RL)^{1/2}/8$.
Following the argument leading from \eqref{e43n} to \eqref{e43p} with $(RL/4)^{1/2}$ replacing
$(RL)^{1/2}$, we find that
\begin{equation}\label{e43u}
\mathbb{P}_y\big(\tau({\hat L}) \leq \tau(C_R)\big)\le \pi^{1/2}j_0\bigg(1+\frac{4R^{1/2}}{j_0L^{1/2}}\bigg)
e^{-j_0L^{1/2}/(4R^{1/2})}.
\end{equation}
This, together with \eqref{e43t}, shows that $\lim_{L\rightarrow\infty}B_2=0$. It remains to obtain the
asymptotic behaviour of $B_1$.

\medskip\noindent
{\bf 5.}
We write $B_1=B_3+B_4+B_5$, where
\begin{equation*}%\label{e43v}
B_3=\frac{1}{\pi R^2L}\int_{\tilde{C}_{L,R}}dy\, \,
\mathbb{E}_y \left[{\textbf{1}}_{\{\mathfrak{B}(C_{R})\cap \hat{L}=\emptyset\}}
\int_{C_{L,R}}dx\, \Big(v_{C_{R}}(x)-v_{C_{R}-\mathfrak{B}(C_{R})}(x)\Big)\right],
\end{equation*}
\begin{equation*}%\label{e43w}
B_4=\frac{1}{\pi R^2L}\int_{\tilde{C}_{L,R}}dy\, \,
\mathbb{E}_y \left[{\textbf{1}}_{\{\mathfrak{B}(C_{R})\cap \hat{L}=\emptyset\}}
\int_{C_{L,R}}dx\, \Big(v_{C_{L,R}}(x)-v_{C_{R}}(x)\Big)\right],
\end{equation*}
\begin{equation}\label{e43x}
B_5=\frac{1}{\pi R^2L}\int_{\tilde{C}_{L,R}}dy\, \,
\mathbb{E}_y \left[{\textbf{1}}_{\{\mathfrak{B}(C_{R})\cap \hat{L}=\emptyset\}}
\int_{C_{L,R}}dx\, \Big(v_{C_{R}-\mathfrak{B}(C_R)}(x)-v_{C_{L,R}-\mathfrak{B}(C_R)}(x)\Big)\right].
\end{equation}
We have that
\begin{align}\label{e43xalt}
B_4
&=\frac{1}{\pi R^2L}\int_{\tilde{C}_{L,R}}dy\, \,
\mathbb{P}_y \big(\{\mathfrak{B}(C_{R})\cap \hat{L}=\emptyset\}\big)
\int_{C_{L,R}}dx\, \Big(v_{C_{L,R}}(x)-v_{C_{R}}(x)\Big)\nonumber\\
&\ge \frac{1}{\pi R^2L}\int_{\tilde{C}_{L,R}}dy\,
\big(\mathcal{T}(C_{L,R})-L\mathcal{T}'(D_R)\big)\nonumber \\&
\ge -\frac{4}{\pi^{1/2}}\int_{[0,\infty)}dt\ t^{1/2}Q_{D_R}'(t),
\end{align}
where we have used the lower bound in \eqref{e22} for $\Omega=D_R$. Furthermore,
\begin{equation}\label{e43y}
B_3=\frac{1}{\pi R^2L}\int_{\tilde{C}_{L,R}}dy\, \,
\mathbb{E}_y \left[\int_{C_{R}}dx\,
\Big(v_{C_{R}}(x)-v_{C_{R}-\mathfrak{B}(C_{R})}(x)\Big)\right]-A_2-A_3,
\end{equation}
where
\begin{align*}%\label{e43z}
A_2
&=\frac{1}{\pi R^2L}\int_{\tilde{C}_{L,R}}dy\, \,
\mathbb{E}_y \left[{\textbf{1}}_{\{\mathfrak{B}(C_{R})\cap \hat{L}=\emptyset\}}
\int_{C_{R}-C_{L,R}}dx\, \Big(v_{C_{R}}(x)-v_{C_{R}-\mathfrak{B}(C_{R})}(x)\Big)\right],
\end{align*}
\begin{equation}\label{f1}
A_3=\frac{1}{\pi R^2L}\int_{\tilde{C}_{L,R}}dy\, \,
\mathbb{E}_y \left[{\textbf{1}}_{\{\mathfrak{B}(C_{R})\cap \hat{L}\ne\emptyset\}}
\int_{C_{R}}dx\, \Big(v_{C_{R}}(x)-v_{C_{R}-\mathfrak{B}(C_{R})}(x)\Big)\right].
\end{equation}

\medskip\noindent
{\bf 6.}
To bound $A_2$ we note that $x \mapsto v_{C_{R}}(x)-v_{C_{R}
-\mathfrak{B}(C_{R})}(x)$ is harmonic on $C_{R}-\mathfrak{B}(C_{R})$, equals $0$ for $x \in
\partial C_R$, and equals $\tfrac14(R^2-|x'|^2)$ for $x\in \mathfrak{B}(C_{R})$. Therefore
\begin{equation*}%\label{f2}
v_{C_{R}}(x)-v_{C_{R}-\mathfrak{B}(C_{R})}(x)
\leq \frac{R^2}{4}\bar{\mathbb{P}}_x\Big(\bar{\tau}(\mathfrak{B}(C_{R}))\le \bar{\tau}(C_R)\Big).
\end{equation*}
On the set $\{\mathfrak{B}(C_{R})\cap \hat{L}=\emptyset\}$ we have that $\bar{\tau}(\hat{L})\le
\bar{\tau}(\mathfrak{B}(C_{R}))$. Hence
\begin{align}\label{f3}
A_2&\le \frac{1}{4\pi L}\int_{\tilde{C}_{L,R}}dy\, \,
\mathbb{E}_y \left[{1}_{\{\mathfrak{B}(C_{R})\cap \hat{L}=\emptyset\}}
\int_{C_{R}-C_{L,R}}dx\,\bar{\mathbb{P}}_x\big(\bar{\tau}(\hat{L})\le \bar{\tau}(C_R)\big)\right]\nonumber\\
&\le \frac{1}{4\pi L}\int_{\tilde{C}_{L,R}}dy\, \,
\mathbb{E}_y \left[\int_{C_{R}-C_{L,R}}dx\,\bar{\mathbb{P}}_x\big(\bar{\tau}(\hat{L})\le \bar{\tau}(C_R)\big)\right]\nonumber\\
&=\frac{R^2}{4}\bigg(1-\frac{2R^{1/2}}{L^{1/2}}\bigg)\int_{C_{R}-C_{L,R}}dx\,
\bar{\mathbb{P}}_x\big(\bar{\tau}(\hat{L})\le \bar{\tau}(C_R)\big).
\end{align}
Recall that $\bar{\tau}(\hat{L})$ equals the first hitting time of $\hat{L}$ by $\bar{\beta}_1$, and that $\bar{\tau}(C_R)$
is the first exit time of $D_R$ by $\bar{\beta}'$. Furthermore, for $x\in C_R-C_{L,R}$ the distance from $x$
to $\hat{L}$ is equal to $(RL/4)^{1/2}+x_1$. By \eqref{e43o},
\begin{equation*}%\label{f4}
\bar{\mathbb{P}}_{x'}\big(\bar{\tau}'(D_R)>\tau\big)\le \frac{R\,e^{-j_0^2\tau/(2R^2)}}{(4\tau)^{1/2}}.
\end{equation*}
It is well known that
\begin{align*}%\label{f5}
\bar{\mathbb{P}}_0^{(1)}\left(\max_{0\le s\le \tau}\bar{\beta}_1(s)>R\right)
=(\pi \tau)^{-1/2}\int_{[R,\infty)}d\xi\, e^{-\xi^2/(4\tau)}
\le 2^{1/2}\,e^{-R^2/(8\tau)}.
\end{align*}
Hence
\begin{equation*}%\label{f6}
\bar{\mathbb{P}}_0^{(1)}\left(\max_{0\le s\le \tau}\bar{\beta}_1(s)>(RL/4)^{1/2}+x_1\right)
\leq 2^{1/2}e^{-(RL+4x_1^2)/(32\tau)}.
\end{equation*}
By the independence of $\bar{\beta}_1$ and $\bar{\beta}'$ we have, similarly to \eqref{e43n},
\begin{align*}%\label{f7}
\bar{\mathbb{P}}_x\big(\bar{\tau}(\hat{L})\le \bar{\tau}(C_R)\big)
&\le 2^{1/2}\int_{[0,\infty)} d\tau\, \bigg(\frac{\partial}{\partial \tau}
\bar{\mathbb{P}}_{x'}\big(\bar{\tau}'(D_R)>\tau\big)\bigg)e^{-(RL+4x_1^2)/(32\tau)}\nonumber \\
&\le \frac{R(RL+4x_1^2)}{2^{11/2}}\int_{[0,\infty)} \frac{d\tau}{\tau^{5/2}}\,
e^{-j_0^2\tau/(2R^2)-(RL+4x_1^2)/(32\tau)}\nonumber\\
&= \frac{R(RL+4x_1^2)}{2^{11/2}} \int_{[0,\infty)} d\tau\, \tau^{1/2}\,
e^{-j_0^2/(2R^2\tau)-(RL+4x_1^2)\tau/32}\nonumber\\
&= \frac{R(RL+4x_1^2)}{2^{9/2}} \int_{[0,\infty)} d\tau\, \tau^2\,
e^{-j_0^2/(2R^2\tau^2)-(RL+4x_1^2)\tau^2/32}\nonumber\\
&= 2\pi^{1/2}R(RL+4x_1^2)^{-1/2}\left(1+\frac{j_0(RL+4x_1^2)^{1/2}}{4R}\right)\,
e^{-j_0(RL+4x_1^2)^{1/2}/(4R)}\nonumber \\
&\le 2\pi^{1/2}\bigg(\frac{R^{1/2}}{L^{1/2}}+\frac{j_0}{4}\bigg)\,
e^{-(j_0^2L/(32R))^{1/2}-(j_0x_1^2/(32R^2))^{1/2}},
\end{align*}
where we have used \cite[3.472.2]{GR}. Integration of the above over $x \in C_R-C_{L,R},$
together with \eqref{f3}, gives
\begin{equation}\label{f8}
A_2=O\big(e^{-(L/(6R))^{1/2}}\big),\qquad L\rightarrow\infty.
\end{equation}

\medskip\noindent
{\bf 7.}
To bound $A_3$ in \eqref{f1}, we use the Cauchy-Schwarz inequality to estimate
\begin{equation}\label{f9}
A_3 \leq \frac{1}{\pi R^2L}\int_{\tilde{C}_{L,R}}dy\, \,
\bigg(\mathbb{P}_y\big(\theta(\hat{L}) \leq
\tau(C_R)\big)\bigg)^{1/2}\bigg(\mathbb{E}_y\left[\int_{C_{R}}dx\,
\Big(v_{C_{R}}(x)-v_{C_{R}-\mathfrak{B}(C_{R})}(x)\Big)\right]^2\bigg)^{1/2}.
\end{equation}
The probability in \eqref{f9} decays sub-exponentially fast in $(L/R)^{1/2}$ by \eqref{e43u}. Hence
it remains to show that the expectation in \eqref{f9} is finite. Define
\begin{equation*}%\label{f10}
\hat{\mathfrak{B}}(C_R) = \left\{x\in C_R\colon\,\min_{0\le s\le \tau(C_R)}\beta_1(s)<x_1
<\max_{0\le s\le \tau(C_R)}\beta_1(s)\right\}.
\end{equation*}
Then $\mathfrak{B}(C_R)\subset\hat{\mathfrak{B}}(C_R)$, and
\begin{align*}%\label{f11}
\mathbb{E}_y\left[\int_{C_{R}}dx\, \Big(v_{C_{R}}(x)-v_{C_{R}-\mathfrak{B}(C_{R})}(x)\Big)\right]^2
\leq\mathbb{E}_y\left[\int_{C_{R}}dx\,
\Big(v_{C_{R}}(x)-v_{C_{R}-\hat{\mathfrak{B}}(C_{R})}(x)\Big)\right]^2.
\end{align*}
For  $x\in \hat{\mathfrak{B}}(C_{R})$ we have $v_{C_R}(x)\le R^2/4$ and $v_{C_R-\hat{\mathfrak{B}}
(C_R)}(x)=0$. Furthermore,
\begin{equation*}%\label{f12}
v_{C_{R}}(x)-v_{C_{R}-\mathfrak{B}(C_{R})}(x)
\le \frac{R^2}{4}\bar{\mathbb{P}}_x\big(\bar{\tau}(\hat{\mathfrak{B}}(C_{R}))
\le \bar{\tau}(C_R)\big),\qquad x\in C_R-\hat{\mathfrak{B}}(C_R),
\end{equation*}
and hence
\begin{align}\label{f13}
\mathbb{E}_y
&\left[\int_{C_{R}}dx\, \Big(v_{C_{R}}(x)-v_{C_{R}-\hat{\mathfrak{B}}(C_{R})}(x)\Big)\right]^2\nonumber\\
&\le \frac{R^4}{8}\mathbb{E}_y\left[|\hat{\mathfrak{B}}(C_{R})|^2
+\bigg(\int_{C_R-\hat{\mathfrak{B}}(C_{R})}dx\,
\bar{\mathbb{P}}_x\Big(\bar{\tau}(\hat{\mathfrak{B}}(C_{R}))\le \bar{\tau}(C_R)\Big)\bigg)^2\right].
\end{align}
The probability distribution of the range of one-dimensional Brownian motion is known (see, for
example, \cite[Eq. (19)]{DSY}). This gives
\begin{equation}\label{f14}
\mathbb{E}_{y'}\left[\max_{0\le s\le \tau'(D_R)}\beta_1(s)-\min_{0\le s\le \tau'(D_R)}
\beta_1(s)\right]^2 = \frac{64\log 2}{\pi^{1/2}}\,\tau'(D_R).
\end{equation}
By a calculation similar to \eqref{e35} and \eqref{e36}, we see that
\begin{align*}%\label{f15}
\mathbb{E}_{y'}\left[\max_{0\le s\le \tau'(D_R)}
\beta_1(s)-\min_{0\le s\le \tau'(D_R)}\beta_1(s)\right]^2
&=\frac{64\log 2}{\pi^{1/2}}\int_{[0,\infty)} d\tau\, \int_{D_R}dz'\, p'_{D_R}(y',z';\tau)\nonumber\\
&=\frac{64\log 2}{\pi^{1/2}}\,v'_{D_R}(y') \leq \frac{16\log 2}{\pi^{1/2}}R^2.
\end{align*}
Together with \eqref{f14}, this yields
\begin{equation*}%\label{f16}
\mathbb{E}_y\big(|\hat{\mathfrak{B}}(C_{R})|^2\big)\le 16\pi^{3/2}(\log 2)R^6,
\end{equation*}
which gives us control over the first term in the right-hand side of \eqref{f13}. To estimate the
second term in the right-hand side of \eqref{f13}, we note that the set $C_R-\hat{\mathfrak{B}}(C_R)$
consists of two semi-infinite cylinders. It is instructive to calculate this term explicitly. To simplify
notation, we define $C_R^+=\{x\in\R^3\colon\, x_1>0, |x'|<R\},\,Z_R=\{x\in \R^3\colon\,x_1
=0, |x'|\le R\},$ and $\vartheta(Z_R)=\inf\{s\ge 0\colon \bar{\beta}(s)\in Z_R\}$. Then, by separation of variables and integration
by parts, we get
\begin{align}\label{f17}
\bar{\mathbb{P}}_x\big(\vartheta(Z_R)\le \bar{\tau}(C_R^+)\big)
&=\int_{[0,\infty)} \bar{\mathbb{P}}_{x'}\big(\bar{\tau}'(D_R)\in d\tau\big)\bar{\mathbb{P}}_{x_1}\big(\vartheta(Z_R)\le \tau\big)\nonumber \\
&= \int_{[0,\infty)} \bar{\mathbb{P}}_{x'}\big(\bar{\tau}'(D_R)\in d\tau\big)\,\frac{2}{\pi^{1/2}}
\int_{[x_1/(2\tau^{1/2}),\infty)}d\xi\,e^{-\xi^2}\nonumber \\
&= \int_{[0,\infty)} d\tau\, \bar{\mathbb{P}}_{x'}\big(\bar{\tau}'(D_R)>\tau\big)\frac{2x_1}{\pi\tau^{3/2}}
\,e^{-x_1^2/(4\tau)}.
\end{align}
Integrating \eqref{f17} with respect to
$x_1 \in \R^+$, we find that
\begin{equation}\label{f18}
\int_{\R^+} dx_1 \bar{\mathbb{P}}_x\big(\vartheta(Z_R)\le \bar{\tau}(C_R^+)\big)
=\frac{4}{\pi^{1/2}} \int_{[0,\infty)} d\tau\,\tau^{-1/2}\,\bar{\mathbb{P}}_{x'}\big(\bar{\tau}'(D_R)>\tau\big).
\end{equation}
Subsequently integrating both sides of \eqref{f18} over $x' \in D_R$, we get
\begin{equation*}%\label{f19}
\int_{C_R^+} dx\, \bar{\mathbb{P}}_x\big(\vartheta(Z_R)\le \bar{\tau}(C_R^+)\big)
=\frac{4}{\pi^{1/2}}\int_{[0,\infty)} d\tau\, \tau^{-1/2}Q'_{D_R}(\tau).
\end{equation*}
It follows that
\begin{equation}\label{f20}
\bigg(\int_{C_R-\hat{\mathfrak{B}}(C_{R})}dx\, \bar{\mathbb{P}}_x\big(\bar{\tau}(\hat{\mathfrak{B}}(C_{R}))
\le \bar{\tau}(C_R)\big)\bigg)^2 = \frac{64}{\pi}\bigg(\int_{[0,\infty)}d\tau\, \tau^{-1/2}Q'_{D_R}(\tau)\bigg)^2.
\end{equation}
The integral over $\tau$ in \eqref{f20} is finite by \eqref{e40}.
We conclude that, by \eqref{e43u},
\begin{align}\label{f21}
A_3
&\le \bigg(\mathbb{P}_y\big(\theta(\hat{L})\le \tau(C_R)\big)\bigg)^{1/2}
\bigg(2\pi^{3/2}(\log 2)R^{10} + \frac{8}{\pi}R^4\bigg(\int_{[0,\infty)}
d\tau\,\tau^{-1/2}Q'_{D_R}(\tau)\bigg)^2\bigg)^{1/2}\nonumber\\
&=O\big(e^{-j_0L^{1/2}/(4R^{1/2})}\big),\qquad L\rightarrow\infty.
\end{align}

\medskip\noindent
{\bf 8.}
The integrand in \eqref{e43y} is independent of $y_1$. Since $\lim_{L\rightarrow\infty}
(L-2(RL)^{1/2})/L=1$, we have by \eqref{e43y}, \eqref{f8} and \eqref{f21} that
\begin{equation}\label{f22}
\liminf_{L\rightarrow\infty} B_3 \geq \frac{1}{\pi R^2}\int_{D_R}dy'\,
\mathbb{E}_{(0,y')}\left[ \int_{C_R}dx\,\Big(v_{C_R}(x)-v_{C_R-\mathfrak{B}(C_R)}(x)\Big)\right].
\end{equation}

\medskip\noindent
{\bf 9.}
It remains to obtain a lower bound on $B_5$ in \eqref{e43x} as $L\rightarrow\infty$. The
integrand with respect to $x$ is a non-negative harmonic function, which can be bounded from
below by enlarging the set $\mathfrak{B}(C_R)$ to $\hat{C}_{R,L}:= \{D_R\times [-\frac{L}{2}+\frac12(RL)^{1/2},
\frac{L}{2}-\frac12(RL)^{1/2}] \}.$ Hence
\begin{align}\label{f23}
B_5
&\geq \frac{1}{\pi R^2L}\int_{\tilde{C}_{L,R}}dy\, \,
\mathbb{E}_y \left[{\textbf{1}}_{\{\mathfrak{B}(C_{R})\cap \hat{L}=\emptyset\}}
\int_{C_{L,R}}dx\, \Big(v_{C_{R}-\hat{C}_{R,L}}(x)-v_{C_{L,R}-\hat{C}_{R,L}}(x)\Big)\right]\nonumber\\
&=\frac{1}{\pi R^2L}\int_{\tilde{C}_{L,R}}dy\,\mathbb{P}_y \big(\mathfrak{B}(C_{R})
\cap \hat{L}=\emptyset\big)\int_{C_{L,R}-\hat{C}_{R,L}}dx\,
\Big(v_{C_{R}-\hat{C}_{R,L}}(x)-v_{C_{L,R}-\hat{C}_{R,L}}(x)\Big).
\end{align}
The set $C_{L,R}-\hat{C}_{R,L}$ consists of two cylinders with cross-section $D_R$ and length
$(RL)^{1/2}/2$ each. Hence, by Theorem \ref{the1.1}(i), we have
\begin{equation}\label{f24}
\int_{C_{L,R}-\hat{C}_{R,L}}dx\,v_{C_{L,R}-\hat{C}_{R,L}}(x)
= \mathcal{T}'(D_R)(RL)^{1/2}-\frac{8}{\pi^{1/2}}\int_{[0,\infty)}dt\ t^{1/2} Q_{D_R}'(t)+O(L^{-1/2}).
\end{equation}
The set $C_R-\hat{C}_{R,L}$ consists of two semi-infinite cylinders, and we integrate the torsion function
for that set over two cylinders of length $(RL)^{1/2}/2$, each near their base. Adopting previous
notation, we get
\begin{align}\label{f25}
&\int_{C_{L,R}-\hat{C}_{R,L}}dx\,v_{C_{R}-\hat{C}_{R,L}}(x)
= 2\int_{[0,(RL)^{1/2}/2)}dx_1 \int_{D_R}dx'v_{C_R^+}(x)\nonumber \\
&= 2\int_{[0,\infty)}dt\int_{[0,(RL)^{1/2}/2)}dx_1 \int_{D_R}dx'\int_{[0,\infty)}
dx_1 \int_{D_R}dy'\int_{[0,\infty)} dy_1\,p_{D_R}'(x',y';t)p_{\R^+}(x_1,y_1;t)\nonumber \\
&=2\int_{[0,\infty)}dt\int_{[0,(RL)^{1/2}/2)} dx_1\, u_{\R^+}(x_1;t)  Q'_{D_R}(t)\nonumber\\
&=2\int_{[0,\infty)}dt\int_{[0,(RL)^{1/2}/2)} dx_1\, \bigg(1-\frac{2}{\pi^{1/2}}
\int_{[x_1/(4t)^{1/2},\infty)} d\xi\, e^{-\xi^2}\bigg)  Q'_{D_R}(t)\nonumber \\
&=\mathcal{T}'(D_R)(RL)^{1/2}-\frac{4}{\pi^{1/2}} \int_{[0,\infty)}dt\,Q'_{D_R}(t)\,
\int_{[0,(RL)^{1/2}/2)}dx_1\int_{[x_1/(4t)^{1/2},\infty)} d\xi\, e^{-\xi^2} \nonumber \\
&\ge \mathcal{T}'(D_R)(RL)^{1/2}-\frac{4}{\pi^{1/2}}\int_{[0,\infty)} dt\,Q'_{D_R}(t)\,
\int_{[0,\infty)}dx_1\int_{[x_1/(4t)^{1/2},\infty)}d\xi\, e^{-\xi^2}\nonumber \\
&=\mathcal{T}'(D_R)(RL)^{1/2}-\frac{4}{\pi^{1/2}}\int_{[0,\infty)}dt\, t^{1/2}\,Q'_{D_R}(t).
\end{align}
Combining \eqref{f23}, \eqref{f24} and \eqref{f25}, we arrive at
\begin{equation*}%\label{f26}
B_5 \ge \frac{1}{\pi R^2L}\int_{\tilde{C}_{L,R}}dy\,
\mathbb{P}_y \big(\mathfrak{B}(C_{R})\cap \hat{L}=\emptyset\big)
\bigg(\frac{4}{\pi^{1/2}}\int_{[0,\infty)}dt\, t^{1/2}\,Q'_{D_R}(t)+O(L^{-1/2})\bigg).
\end{equation*}
We conclude that
\begin{equation}\label{f27}
\liminf_{L\rightarrow\infty}B_5 \geq \frac{4}{\pi^{1/2}}\int_{[0,\infty)} dt\, t^{1/2}\,Q'_{D_R}(t).
\end{equation}

\medskip\noindent
{\bf 10.}
From \eqref{e43xalt}, \eqref{f22} and \eqref{f27}, we get
\begin{equation*}%\label{f28}
\liminf_{L\rightarrow\infty}(B_3+B_4+B_5) \geq \frac{1}{\pi R^2}\int_{D_R}dy'\,
\mathbb{E}_{(0,y')}\left[ \int_{C_R}dx\,\Big(v_{C_R}(x)-v_{C_R-\mathfrak{B}(C_R)}(x)\Big)\right].
\end{equation*}
Scaling each the space variables $y'$ and $x$ by a factor $R$, we gain a factor $R^5$ for the
respective integrals with respect to $y'$ and $x$. Furthermore, scaling the torsion functions
$v_{C_R}$ and $v_{C_R-\mathfrak{B}(C_R)}$, we gain a further factor $R^2$. Hence
\begin{align*}%\label{e44}
&\frac{1}{\pi R^2}\int_{D_R}dy'\,\mathbb{E}_{(0,y')}\left[
\int_{C_R}dx\,\Big(v_{C_R}(x)-v_{C_R-\mathfrak {B}(C_R)}(x)\Big)\right]\nonumber \\
&\qquad =\frac{1}{\pi}R^5\int_{D_1}dy'\,\mathbb{E}_{(0,y')}\left[
\int_{C_1}dx\,\Big(v_{C_1}(x)-v_{C_1-\mathfrak {B}(C_1)}(x)\Big)\right],
\end{align*}
which is the required first formula in \eqref{e43}.
\end{proof}

The main modification for the proof for $\mathfrak{C}(C_{L,R})$ in the second formula of \eqref{e43}
is that no averaging takes place over the cross-section $D_R$ as $y'=0$ is fixed. Hence the absence
of the factor $\frac{1}{\pi}$ and the integral with respect to $y'$ over $D_1$ in the formula for $c'$ in \eqref{e43ext}.

%%%%%%%% SECTION 4 %%%%%%%%%%%%%%%%%%%%%%%%

\section{Proofs of Theorems \ref{the1.2} and \ref{the1.3}}\label{sec4}

The proofs of Theorems \ref{the1.2} and \ref{the1.3} are given in Section \ref{subsec4.1}
and \ref{subsec4.2}, respectively, and rely on Proposition \ref{prop1}.

%%%

\subsection{Proof of Theorem \ref{the1.2}}\label{subsec4.1}

To prove the upper bound we note that $\lambda_1'(D_R)=j_0^2/R^2$ and $\mathcal{T}'(D_R)= \pi R^4/8$ (see \cite{vdBBB}).
This gives the upper bound $\pi R^5/2j_0$ for the right-hand side of \eqref{e24}, which implies
the upper bound for $c$ in \eqref{e25a}.

To prove the lower bound we start from \eqref{e43ext}. Let $a\in (0,\frac14)$. We have the following estimate:
\begin{align}\label{e45}
&c = \frac{1}{\pi}\int_{D_1}dy'\,\mathbb{E}_{(0,y')}\left[
\int_{C_1} dx\,\Big(v_{C_1}(x)-v_{C_1-\mathfrak {B}(C_1)}(x)\Big)\right]\nonumber \\
&\qquad\geq \frac{1}{\pi}\int_{D_a}dy'\, \mathbb{E}_{(0,y')}\left[\int_{C_1}dx\,
\Big(v_{C_1}(x)-v_{C_1-\mathfrak {B}(C_1)}(x)\Big)\right]\nonumber \\
&\qquad \geq \frac{1}{\pi}\int_{D_a}dy'\,
\mathbb{E}_{(0,y')}\left[\int_{\{x\in\R^3\colon |x-\beta(0)|<a\}}dx\,
\Big(v_{C_1}(x)-v_{C_1-\mathfrak {B}(B(\beta(0);a))}(x)\Big)\right],
\end{align}
where we have used that $\mathfrak{B}(C_1)\supset \mathfrak {B}(B((0,y');a))$. To estimate the
second integral, we consider a fixed compact set $K\subset B((0,y');a) \subset \R^3$ and
derive a lower bound for $v_{C_1}(x)-v_{C_1-K}(x)$ uniformly in $|y'| \leq a$ and $|x-(0,y')|
\leq a$.

First note that $x \mapsto v_{C_1}(x)-v_{C_1-K}(x)$ is harmonic on $C_1-K$, equals $0$
for $x \in \partial C_1$, and equals $\tfrac14(1-|x'|^2)$ for $x\in K$. If $|y'|<a$, then $|x'|<2a$,
$x\in K$. Hence $v_{C_1}(x)-v_{C_1-K}(x) \geq \tfrac14(1-4a^2)$ for $x \in K$. We therefore
have
\begin{equation}\label{e48}
v_{C_1}(x)-v_{C_1-K}(x)\ge \frac{1-4a^2}{4}\,\bar{\mathbb{P}}_x\big(\bar{\tau}_{\R^3-K}<\bar{\tau}(C_1)\big),
\qquad x \in C_1.
\end{equation}
By the strong Markov property, we have
\begin{align}\label{e49}
\bar{\mathbb{P}}_x\big(\bar\tau_{\R^3-K}<\bar\tau(C_1)\big)
&=\bar{\mathbb{P}}_x\big(\bar\tau_{\R^3-K}<\infty\big)-\bar{\mathbb{P}}_x\big(\bar\tau(C_1)\le\bar\tau_{\R^3-K}<\infty\big)\nonumber \\
&\ge \inf_{\{|x-(0,y')|<a\}}\bar{\mathbb{P}}_x\big(\bar\tau_{\R^3-K}<\infty\big)-\sup_{x\in \partial C_1}
\bar{\mathbb{P}}_x\big(\bar\tau_{\R^3-K}<\infty\big).
\end{align}
Let $\mu_{K}$ denote the equilibrium measure for $K$. Then (see \cite{PS})
\begin{equation}\label{e50}
\bar{\mathbb{P}}_x\big(\bar\tau_{\R^3-K}<\infty\big) = \int_K \mu_{K}(dz)\,\frac{1}{4\pi|x-z|}, \qquad x \in K.
\end{equation}
If $z\in K$ and $|x-(0,y')| \le a$, then $|x-z|\le 2a$. Hence \eqref{e50} gives
\begin{equation}\label{e51}
\inf_{\{|x-(0,y')|<a\}} \bar{\mathbb{P}}_x\big(\bar\tau_{\R^3-K}<\infty\big) \geq
\frac{1}{8\pi a} \int_K \mu_K(dz) = \frac{1}{8\pi a}\,\textup{cap}(K).
\end{equation}
Furthermore, if $x\in \partial C_1$, $z\in K$ and $|y'|\leq a$, then $|z-x|\ge 1-2a$. Hence
\eqref{e50} also gives
\begin{equation}\label{e52}
\sup_{x\in \partial C_1} \bar{\mathbb{P}}_x\big(\bar\tau_{\R^3-K}<\infty\big) \leq \frac{1}{4\pi(1-2a)}\,\textup{cap}(K).
\end{equation}
Combining \eqref{e51} and \eqref{e52}, we get
\begin{equation}\label{e53}
\bar{\mathbb{P}}_x\big(\bar\tau_{\R^3-K}<\bar{\tau}(C_1)\big) \geq \frac{1-4a}{8\pi a(1-2a)}\,\textup{cap}(K),
\qquad K \subset B((0,y');a),\, |x-(0,y')| \leq a\,\ |y'| \leq a.
\end{equation}
Combining \eqref{e45}, \eqref{e48} and \eqref{e53}, we arrive at
\begin{align}\label{e54}
c
&\geq \frac{1-4a^2}{4\pi}\,\,\frac{1-4a}{8\pi a(1-2a)}
\int_{D_a}dy'\, \int_{\{x\in \R^3\colon\,|x-(0,y')|<a\}} dx\,
\mathbb{E}_{(0,y')}\big[\textup{cap}\big(\mathfrak {B}(B(\beta(0);a))\big)\big]\nonumber \\
&=\frac{(1-4a)(1+2a)a^4}{24}\,
\mathbb{E}_0\big[\textup{cap}\big(\mathfrak {B}(B(0;a))\big)\big]\nonumber \\ &
=\frac{(1-4a)(1+2a)a^5}{24}\,\kappa,
\end{align}
where we have used that $\mathcal{H}^2(D_a)=\pi a^2$, $|B(0;a)| = \frac{4\pi}{3}a^3,$ and
\begin{equation*}%\label{e54ext}
\mathbb{E}_0\big[\textup{cap}\big(\mathfrak{B}(B(0;a))\big)\big]
= a\,\mathbb{E}_0\big[\textup{cap}\big(\mathfrak{B}(B(0;1))\big)\big] = \kappa a.
\end{equation*}
The right-hand side of \eqref{e54} is maximal when
\begin{equation*}%\label{e55}
a=\frac{\sqrt{79}-3}{28}.
\end{equation*}
This choice of $a$ yields the left-hand side of \eqref{e25a}. \qed

%%%

\subsection{Proof of Theorem \ref{the1.3}}\label{subsec4.2}

We first prove the upper bound. By \eqref{e35},
\begin{equation}\label{e65}
\mathbb{E}_{0}\big[\tau'(D_R)^{1/2}\big] = \int_{[0,\infty)} d\tau\, \tau^{1/2}\,
\mathbb{P}_{y'}\big(\tau'(D_R)\in d\tau\big)
= \frac12 \int_{[0,\infty)} d\tau\,\tau^{-1/2} \int_{D_R} dz'\, p_{D_R}'(0,z';\tau).
\end{equation}
By the monotonicity of the Dirichlet heat kernel,
\begin{equation}\label{e66}
p_{D_R}'(0,z';\tau)\le p'  _{\R^2}(0,z';\tau)=(4\pi \tau)^{-1}e^{-|z'|^2/(4\tau)}.
\end{equation}
Combining \eqref{e65} and \eqref{e66}, we get
\begin{equation}\label{e67}
\mathbb{E}_{0}\big[\tau'(D_R)^{1/2}\big]
\le \frac12 \int_{[0,\infty)} d\tau\,\tau^{-1/2} \int_{D_R} dz'\,(4\pi \tau)^{-1}\,e^{-|z'|^2/(4\tau)}
= \tfrac12 \pi^{1/2} R.
\end{equation}
Combining \eqref{e33}, \eqref{e34} and \eqref{e67}, we obtain
\begin{align}\label{e68}
\mathbb{E}_{0}\left[\mathcal{T}(C_{L,R}-\mathfrak{B}(C_{L,R}))\right] \ge
\big(L-2R)\mathcal{T}'(D_R) - \frac{8}{\pi^{1/2}}\int_{[0,\infty)}dt\ t^{1/2}Q_{D_R}'(t).
\end{align}
From \eqref{e22} we have
\begin{equation}\label{e69}
\mathcal{T}(C_{L,R}) \le \mathcal{T}'(D_R)L-\frac{4}{\pi^{1/2}}\int_{[0,\infty)}dt\ t^{1/2}Q_{D_R}'(t)
+\frac{8}{L\lambda_1'(D_R)}\mathcal{T}'(D_R).
\end{equation}
Combining \eqref{e25c}, \eqref{e68} and \eqref{e69}, we get
\begin{equation*}%\label{e70}
\mathfrak{C}(C_{L,R})
\le 2R\mathcal{T}'(D_R)+\frac{4}{\pi^{1/2}}\int_{[0,\infty)}dt\ t^{1/2}Q_{D_R}'(t)
+\frac{8}{L\lambda_1'(D_R)}\mathcal{T}'(D_R).\nonumber\\
\end{equation*}
Since $\mathcal{T}'(D_R)=\frac{\pi}{8}R^4$, we conclude by \eqref{e41} with $\Omega=D_R$, that
\begin{align*}%\label{e71}
\limsup_{L\rightarrow\infty} \mathfrak{C}(C_{L,R}) \le \frac{\pi}{4} \left(1+\frac{1}{j_0}\right) R^5.
\end{align*}

\medskip\noindent
To prove the lower bound we start from \eqref{e43ext}. Let $a\in (0,\frac13).$ We have the following estimate:
\begin{equation*}%\label{e58}
c ' = \mathbb{E}_0\left[ \int_{C_1}dx\,\Big(v_{C_1}(x)-v_{C_1-\mathfrak {B}(C_1)}(x)\Big)\right]
\geq \mathbb{E}_0\left[\int_{D_a}dx\,
\Big(v_{C_1}(x)-v_{C_1-\mathfrak {B}(B(\beta(0);a))}(x)\Big)\right].
\end{equation*}
Fix a compact set $K\subset B(B(0);a) \subset \R^3$. Note that $x \mapsto v_{C_1}(x)-v_{C_1-K}(x)$
is harmonic on $C_1-K$, equals $0$ for $x \in \partial C_1$, and equals $\tfrac14(1-|x'|^2)$ for
$x\in K$. If $|x|<a$, then $|x'|<a$, $x\in K$. Hence $v_{C_1}(x)-v_{C_1-K}(x) \geq \tfrac14(1-a^2)$
for $x \in K$. We therefore have
\begin{equation*}%\label{e60}
v_{C_1}(x)-v_{C_1-K}(x) \geq \frac{1-a^2}{4}\,\bar{\mathbb{P}}_x\big(\bar{\tau}_{\R^3-K}<\bar{\tau}(C_1)\big),
\qquad x \in C_1.
\end{equation*}
It is straightforward to check that \eqref{e51} holds for $y'=0$.
 Furthermore, if $x\in \partial C_1$ and $ z\in K$,
then $|z-x| \geq 1-a$. Hence, by \eqref{e50},
\begin{equation*}%\label{e61}
\sup_{x\in \partial C_1}\bar{\mathbb{P}}_x\big(\tau_{\R^3-K}<\infty\big)\le \frac{1}{4\pi(1-a)}\,\textup{cap}(K).
\end{equation*}
Combining \eqref{e51} and \eqref{e52}, we get
\begin{equation*}%\label{e62}
\bar{\mathbb{P}}_x\big(\bar{\tau}_{\R^3-K}<\bar{\tau}(C_1)\big) \geq \frac{1-3a}{8\pi a(1-a)}\,\textup{cap}(K),
\qquad K\subset B(\beta(0);a), \, |x| \leq a.
\end{equation*}
Combining \eqref{e48}, \eqref{e51} and \eqref{e53}, we arrive at
\begin{align}\label{e63}
&\mathbb{E}_0\left[\int_{C_1}dx
\Big(v_{C_1}(x)-v_{C_1-\mathfrak {B}(C_1)}(x)\Big)\right]\nonumber \\
&\geq \frac{1-a^2}{4}\,\frac{1-3a}{8\pi a(1-a)} \int_{\{x\in \R^3\colon |x|<a\}} dx\,
\mathbb{E}_0\big[\textup{cap}\big(\mathfrak {B}(B(0;a))\big)\big]\nonumber \\
&=\frac{(1-3a)(1+a)a^3}{24}\,\kappa.
\end{align}
The right-hand side of \eqref{e63} is maximal when
\begin{equation*}%\label{e64}
a=\frac{\sqrt{61}-4}{15}.
\end{equation*}
This choice of $a$ yields the left-hand side of \eqref{e25e}. \qed

%%%%%%% APPENDIX %%%%%%%%%%%%%%%%%%%

\appendix

\section{Appendix}\label{app}

The following estimate was used in Step 1 of the proof of Proposition \ref{prop1}.

\begin{lemma}\label{lem4}
Let $\O_1\subset\O_2$ be non-empty open sets in $\R^m$ and $K$ a compact set in
$\R^m$. Let the torsion functions for $\O_1,\O_2,\O_1-K,\O_2-K$ be denoted by $v_{\O_1},v_{\O_2},
v_{\O_1-K},v_{\O_2-K}$, respectively. Suppose that $\textup{inf}[\textup{spec}(-\Delta_{\O_2})]>0$.
Then
\begin{equation*}%\label{e72}
v_{\O_2}(x)-v_{\O_2-K}(x)\ge v_{\O_1}(x)-v_{\O_1-K}(x),
\qquad x \in \O_1-K,
\end{equation*}
and
\begin{equation}\label{e73}
v_{\O_2}(x)-v_{\O_2-K}(x)\le  \frac18(m+cm^{1/2}+8)\lambda(\Omega_2)^{-1},
\qquad x \in \O_1-K,
\end{equation}
with
\begin{equation*}%\label{e74}
c = \sqrt{5(4+\log 2)}.
\end{equation*}
\end{lemma}

\begin{proof}
We extend the torsion functions $v_{\O_2-K}$ and  $v_{\O_1-K}$ to all of $\O_1$ by putting them
equal to $0$ on $K\cup (\R^m-\O_1)$. Define $h(x)=(v_{\O_2}(x)-v_{\O_2-K}(x))-(v_{\O_1}(x)
-v_{\O_1-K}(x))$, $x\in \O_1-K.$ Then $h$ is harmonic on $\O_1-K$, and $h(x)=v_{\O_2}(x)-v_{\O_1}(x)
\geq 0$, $x\in K$, by the domain monotonicity of the torsion function. Furthermore, $h(x)=v_{\O_2}(x)
-v_{\O_2-K}(x)\geq 0$, $x\in \partial \O_1$, by the domain monotonicity, and $h(x) \geq 0$, $x\in \O_1-K$,
by the maximum principle of harmonic functions. The estimate in \eqref{e73} follows from the non-negativity
of the torsion function, together with the estimate in \cite{HV}.
\end{proof}

%%%%%%%%%%%%%%%%% REFERENCES %%%%%%%%%%%%%%%%%%

%%%%%%%%%%%%%%%%%%%%%%%%%%%%%%%%%%%%%

\end{document}